\documentclass[12pt,oneside,reqno]{amsart}

\usepackage{amsmath,amssymb,amsthm,textcomp}
\usepackage{amsfonts,graphicx}
\usepackage[mathscr]{eucal}
\usepackage{mathtools}
\usepackage{color}
\usepackage{csquotes}
\usepackage[backend=bibtex,%
firstinits=true,%
doi=false,%
isbn=true,%
url=false,%
maxnames=99]{biblatex}%

\AtEveryBibitem{\clearfield{issn}}
\AtEveryCitekey{\clearfield{issn}}
\numberwithin{equation}{section}
\DeclareNameAlias{sortname}{last-first}
\theoremstyle{definition}

\numberwithin{equation}{section}

\newcommand{\ncom}{\newcommand}

\ncom{\beq}{\begin{equation}}
\ncom{\eeq}{\end{equation}}
\ncom{\bea}{\begin{eqnarray*}}
\ncom{\eea}{\end{eqnarray*}}
\ncom{\beqa}{\begin{eqnarray}}
\ncom{\eeqa}{\end{eqnarray}}
\ncom{\nno}{\nonumber}
\ncom{\non}{\nonumber}
\ncom{\ds}{\displaystyle}
\ncom{\half}{\frac{1}{2}}
\ncom{\mbx}{\makebox{.25cm}}
\ncom{\hs}{\mbox{\hspace{.25cm}}}
\ncom{\rar}{\rightarrow}
\ncom{\Rar}{\Rightarrow}
\ncom{\noin}{\noindent}
\ncom{\bc}{\begin{center}}
\ncom{\ec}{\end{center}}
\ncom{\sz}{\scriptsize}
\ncom{\rf}{\ref}
\ncom{\s}{\sqrt{2}}
\ncom{\sgm}{\sigma}
\ncom{\Sgm}{\Sigma}
\ncom{\psgm}{\sigma^{\prime}}
\ncom{\dt}{\delta}
\ncom{\Dt}{\Delta}
\ncom{\lmd}{\lambda}
\ncom{\Lmd}{\Lambda}
\ncom{\Th}{\Theta}
\ncom{\e}{\eta}
\ncom{\eps}{\epsilon}
\ncom{\pcc}{\stackrel{P}{>}}
\ncom{\lp}{\stackrel{L_{p}}{>}}
\ncom{\dist}{{\rm\,dist}}
\ncom{\sspan}{{\rm\,span}}
\ncom{\re}{{\rm Re\,}}
\ncom{\im}{{\rm Im\,}}
\ncom{\sgn}{{\rm sgn\,}}
\ncom{\ba}{\begin{array}}
\ncom{\ea}{\end{array}}
\ncom{\hone}{\mbox{\hspace{1em}}}
\ncom{\htwo}{\mbox{\hspace{2em}}}
\ncom{\hthree}{\mbox{\hspace{3em}}}
\ncom{\hfour}{\mbox{\hspace{4em}}}
\ncom{\vone}{\vskip 2ex}
\ncom{\vtwo}{\vskip 4ex}
\ncom{\vonee}{\vskip 1.5ex}
\ncom{\vthree}{\vskip 6ex}
\ncom{\vfour}{\vspace*{8ex}}
\ncom{\norm}{\|\;\;\|}
\ncom{\integ}[4]{\int_{#1}^{#2}\,{#3}\,d{#4}}
\ncom{\vspan}[1]{{{\rm\,span}\{ #1 \}}}
\ncom{\dm}[1]{ {\displaystyle{#1} } }
\ncom{\ri}[1]{{#1} \index{#1}}

\newtheorem{theorem}{\bf Theorem}[section]
\newtheorem{remark}{\bf Remark}[section]
\newtheorem{proposition}{Proposition}[section]
\newtheorem{lemma}{Lemma}[section]

\newtheorem{definition}{Definition}[section]
\newtheoremstyle
    {remarkstyle}
    {}
    {11pt}
    {}
    {}
    {\bfseries}
    {:}
    {     }
    {\thmname{#1} \thmnumber{#2} }

\theoremstyle{remarkstyle}

\def\eps{\varepsilon}

\begin{document}
\title{Mixed Fractional Risk Process}
\author[Kuldeep Kumar Kataria]{K. K. Kataria}
\address{Kuldeep Kumar Kataria, Department of Mathematics,
	Indian Institute of Technology Bhilai, Raipur-492015, India.}
\email{kuldeepk@iitbhilai.ac.in}
\author[Mostafizar Khandakar]{M. Khandakar}
\address{Mostafizar Khandakar, Department of Mathematics,
	Indian Institute of Technology Bhilai, Raipur-492015, India.}
\email{mostafizark@iitbhilai.ac.in}
\date{January 26, 2020.}
\begin{abstract}
In this paper, we introduce a risk process, namely, the mixed fractional risk process (MFRP) in which the number of claims in the associated claim process are modelled using the mixed fractional Poisson process (MFPP). The covariance structure of the MFRP is studied and its long-range dependence property has been established. Also, we show that the increments of MFRP exhibit the short-range dependence property.  Another fractional risk process based on the MFPP is introduced which we call as the MFRP-II. It differs from the MFRP in terms of premium. Its ruin probabilities in the case of exponential and subexponential distributions of claim sizes are obtained.
\end{abstract}

\maketitle
\section{Introduction}
The classical collective risk Cram\'er-Lundberg model describes the surplus $R(t)$ of an insurance company till time $t$ as
\begin{equation}\label{9}
R(t)=u+ct-\sum _{i=1}^{N(t)}X_{i}, \quad t\ge0,
\end{equation}
where   $u>0$  is the initial capital and $c>0$ is the premium rate. The number of claims in the above risk process are modelled using homogeneous Poisson process $N(t)$ with intensity parameter $\lambda>0$ whereas its claim sizes  $X_{i}$'s  are  independent and identically distributed (iid) positive random variables with finite mean. The total claim till time  $t$ is given by a compound
Poisson process as $X_{1}+X_{2}+\dots +X_{N(t)}$. To avoid ruin with certainty the net profit condition, that is, $c\,\mathbb{E}\left (W_{i}\right)>\mathbb{E}\left(X_{i}\right)$ is assumed. Here, $W_{i}$ is  exponentially distributed  $i$th interarrival time. Constantinescu {\it et al.} (2018) provided three equivalent expressions for ruin probabilities in the case of gamma-distributed claim sizes in the risk process $\{R(t)\}_{t\ge0}$. The explicit expressions for ruin probabilities in risk models with distributions of claim sizes having rational Laplace transforms are obtained by Constantinescu {\it et al.} (2019).
 
Beghin and Macci (2013) introduced a fractional version of the above risk process as
\begin{equation*}
\bar{R}(t)=u+ct-\sum _{i=1}^{N^{\alpha}(t)}X_{i},  \quad t\ge0,
\end{equation*}
where $N^{\alpha}(t), 0<\alpha\le1$, is the time fractional Poisson process (TFPP). They obtained some asymptotic results for the ruin probabilities of $\bar{R}(t)$ using large deviations for the TFPP. Further, Biard and Saussereau (2014) derived some expressions for the ruin probabilities of $\bar{R}(t)$ in the case of light-tailed and heavy-tailed claim sizes. Recently, Kumar {\it et al.} (2019)   introduced and studied a fractional risk process whose surplus $R^{\alpha}(t)$ is given by
\begin{equation}\label{kumar}
R^{\alpha}(t)=u+\mu(1+\rho)\lambda Y_{\alpha}(t)-\sum_{i=1}^{N^{\alpha}(t)}X_{i}, \ \ t\ge0,
\end{equation}
where $Y_{\alpha}(t)$ is the inverse stable subordinator, $\rho\ge0$ is the safety loading parameter and $\mu$ is the mean of $X_{i}$. They established the long-range dependence (LRD) property of $\{R^{\alpha}(t)\}_{t\ge0}$ and discussed about the average capital required to recover a company after ruin. 

In this paper, we first introduce a compound version of the mixed fractional Poisson process (MFPP) defined as
\begin{equation}\label{qouuvf}
C^{\alpha_{1},\alpha_{2}}(t)\coloneqq\sum_{i=1}^{N^{\alpha_{1},\alpha_{2}}(t)}X_{i},
\end{equation}
where  $X_i$'s are positive integer valued iid random variables with finite mean which are independent of the MFPP $\{N^{\alpha_{1},\alpha_{2}}(t)\}_{t\ge0}$. It is shown that its state probabilities $q^{\alpha_{1},\alpha_{2}}_n(t)=\mathrm{Pr}\{C^{\alpha_{1},\alpha_{2}}(t)=n\}$ satisfy the following fractional differential equation:
\begin{align*}
C_1\partial_t^{\alpha_1}q^{\alpha_{1},\alpha_{2}}_0(t)+C_2\partial_t^{\alpha_2}q^{\alpha_{1},\alpha_{2}}_0(t)&=-\lambda q^{\alpha_{1},\alpha_{2}}_0(t),\\
C_1\partial_t^{\alpha_1}q^{\alpha_{1},\alpha_{2}}_n(t)+C_2\partial_t^{\alpha_2}q^{\alpha_{1},\alpha_{2}}_n(t)&=-\lambda q^{\alpha_{1},\alpha_{2}}_n(t)+\lambda\sum_{i=1}^{n}\mathrm{Pr}\{X_1=i\}q^{\alpha_{1},\alpha_{2}}_{n-i}(t),\ \ n\geq1,
\end{align*}
where $0<\alpha_{2}<\alpha_{1}<1$ and $C_{1}\geq0,C_{2}\geq0$ such that $C_{1}+C_{2}=1$.  Here, $\partial^{\alpha}_{t}$ denotes the Caputo fractional derivative defined as
\begin{equation}\label{plm1}
\partial^{\alpha}_{t}f(t)\coloneqq\begin{cases}
\dfrac{1}{\Gamma \left( 1-\alpha \right)}\displaystyle\int_{0}^{t}(t-s)^{-\alpha}f^{\prime}(s)\,\mathrm{d}s,\ \ 0<\alpha<1,\\
f^{\prime}(t), \ \ \alpha=1.
\end{cases}
\end{equation}
 By identifying the compound MFPP (\ref{qouuvf}) as claim process, we introduce a fractional risk process, namely, the mixed fractional risk process (MFRP) by considering the inverse mixed stable subordinator in the premium received by the insurance company till time $t$. The surplus of the MFRP is 
\begin{equation*}
R^{\alpha_{1},\alpha_{2}}(t)=u+\mu(1+\rho)\lambda Y_{\alpha_{1},\alpha_{2}}(t)-\sum_{i=1}^{N^{\alpha_{1},\alpha_{2}}(t)}X_{i}, \ \ t\ge0.
\end{equation*}
We obtain the covariance of MFRP in terms of the known covariance of the inverse mixed stable subordinator. Using its covariance structure we establish the LRD property of MFRP. Also, we show that the increments of MFRP exhibit the short-range dependence (SRD) property.

Another fractional risk process based on the MFPP is introduced whose surplus is given by 
\begin{equation*}
\bar{R}^{\alpha_{1},\alpha_{2}}(t)=u+ct-\sum_{i=1}^{N^{\alpha_{1},\alpha_{2}}(t)}X_{i}, \ \ t\ge0.
\end{equation*}
We call the process $\{\bar{R}^{\alpha_{1},\alpha_{2}}(t)\}_{t\ge0}$ as MFRP-II. Note that MFRP-II differs from MFRP with respect to premium. Its ruin probabilities are obtained in two different cases of the claim sizes.

The paper is organised as follows: In Section 2, we give some preliminary results on subordinator, inverse subordinator and the mixture of inverse subordinators. In Section 3, we first briefly discuss the MFPP and obtain its one dimensional distributions. Then, we define a compound version of the MFPP and obtain the governing fractional differential equations of its state probabilities.  We show that the MFPP and the introduced compound MFPP exhibits overdispersion. In Section 4, a risk process based on the mixed fractional Poissonian claims, namely, the MFRP is introduced. It is shown that the MFRP is a martingale with respect to a suitable filtration provided the safety loading factor $\rho=0$. We obtain the covariance of MFRP using which it is shown that this fractional risk process exhibits LRD property. Moreover, it is shown that the increments of MFRP exhibits SRD property. In Section 5, a second fractional risk process, namely, the MFRP-II is introduced, and the expressions for its ruin probabilities in the case of exponential and subexponential distributions of claim sizes are obtained.

\section{Preliminaries}
In this section, we present some known results which will be required later.
\subsection{Subordinator and inverse subordinator}
A subordinator $\{D(t)\}_{t\geq0}$ is a one-dimensional L\'evy process whose sample paths are non-decreasing. It is characterized by the following Laplace transform:
\begin{equation*}
\mathbb{E}\left(e^{-sD(t)}\right)=e^{-t\phi(s)},\ \ s\geq0,
\end{equation*}
where ${\phi(s)=\mu s+\displaystyle\int_{0}^{\infty}\left(1-e^{-s x}\right)\,\Pi(\mathrm{d} x)}$ is the L\'evy exponent. Here, $\mu \geq 0$  is the drift and  $\Pi$  is the L\'evy measure. For $\phi(s)=s^\alpha$, $0<\alpha<1$, it reduces to a driftless ($\mu=0$) subordinator known as the $\alpha$-stable subordinator denoted by $\{D_\alpha(t)\}_{t\geq0}$. 

The first-passage time of $\{D(t)\}_{t\geq0}$ is a  non-decreasing process $\{Y(t)\}_{t\geq 0}$ known as the inverse subordinator. It is defined as
\begin{equation*}
Y(t)=\inf \left\{s\geq0:D(s)>t\right\},\ \ t\geq 0.
\end{equation*}
 Similarly, the inverse $\alpha$-stable subordinator $\{Y_\alpha(t)\}_{t\geq0}$ is the first-passage time of $\{D_\alpha(t)\}_{t\geq0}$. For further details on subordinator and its inverse, we refer the reader to Applebaum (2004).

\subsection{Mixture of inverse subordinators}
The process mixed stable subordinator $\left\{D_{\alpha_{1}, \alpha_{2}}(t)\right\}_{t\geq0}$ is characterized by the Laplace transform given by
\begin{equation}\label{3}
\mathbb{E}\left(e^{-s D_{\alpha_{1}, \alpha_{2}}(t)}\right)=e^{-t\left(C_{1} s^{\alpha_{1}}+C_{2} s^{\alpha_{2}}\right)},\ \ s\geq0,
\end{equation}
where  $C_{1}+C_{2}=1$, $C_{1}\geq0$, $C_{2}\geq0$ and $0<\alpha_{2}<\alpha_{1}<1$. It is known that  
\begin{equation}\label{4}
D_{\alpha_{1}, \alpha_{2}}(t)\overset{d}{=}\left(C_{1}\right)^{\frac{1}{\alpha_{1}}} D_{\alpha_{1}}(t)+\left(C_{2}\right)^{\frac{1}{\alpha_{2}}} D_{\alpha_{2}}(t), \quad t \geq 0,
\end{equation}
where $D_{\alpha_{1}}$ and $D_{\alpha_{2}}$ are two independent stable subordinators such that $0<\alpha_2<\alpha_1<1$. Here, $\overset{d}{=}$ means equal in distribution. The inverse mixed stable subordinator $\left\{Y_{\alpha_{1}, \alpha_{2}}(t)\right\}_{t\geq0}$ is defined as follows: 
\begin{equation}\label{5}
Y_{\alpha_{1},\alpha_{2}}(t)=\inf \left\{s\geq0 :D_{\alpha_{1},\alpha_{2}}(s)>t\right\}, \quad t\geq 0.
\end{equation}
For $C_{2}=0$, it reduces to inverse stable subordinator. A similar result holds for $C_{1}=0$.

Leonenko {\it et al.} (2014) obtained the following expression for the expected value of the inverse mixed stable subordinator $U_{{\alpha_{1},\alpha_{2}}}(t)=\mathbb{E}\left(Y_{\alpha_{1},\alpha_{2}}(t)\right)$ :
\begin{equation}\label{7}
U_{{\alpha_{1},\alpha_{2}}}(t)=\frac{t^{\alpha_{1}}}{C_{1}}  E_{\alpha_{1}-\alpha_{2}, \alpha_{1}+1}\left(-C_{2}t^{\alpha_{1}-\alpha_{2}}/C_{1}\right),
\end{equation}
where  $E_{\alpha_{1}-\alpha_{2}, \alpha_{1}+1}(.)$  is the two-parameter Mittag-Leffler function defined as (see Mathai and Haubold (2008)) 
\begin{equation*}
E_{\alpha, \beta}(t)\coloneqq\sum_{k=0}^{\infty} \frac{t^{k}}{\Gamma(k\alpha+\beta)},\ \ \alpha>0,\ \beta>0.
\end{equation*}
It is generalized to three-parameter Mittag-Leffler function which is defined as 
\begin{equation*}
E_{\alpha, \beta}^{\gamma}(t)\coloneqq\sum_{k=0}^{\infty} \frac{\gamma(\gamma+1)\ldots(\gamma+k-1)t^{k}}{k!\Gamma(k\alpha+\beta)},\ \ \alpha>0,\ \beta>0,\ \gamma>0.
\end{equation*}
When $\gamma=1$, it reduces to two-parameter Mittag-Leffler function. For $\lambda>0$ and $\beta\neq\alpha\gamma$ the following asymptotic result holds (see Beghin (2012), Eq. (2.44)):
\begin{equation*}
E_{\alpha, \beta}^{\gamma}\left(-\lambda t^{\alpha}\right)=\frac{\lambda^{-\gamma}t^{-\alpha \gamma}}{\Gamma(\beta-\alpha\gamma)}+o(t^{-\alpha \gamma}),\ \ t\to\infty.
\end{equation*}
It follows from the above equation  that
\begin{equation}\label{3m}
E_{\alpha, \beta}^{\gamma}\left(-\lambda t^{\alpha}\right)\sim\frac{\lambda^{-\gamma}t^{-\alpha \gamma}}{\Gamma(\beta-\alpha\gamma)},\ \ t\to\infty.
\end{equation}
Veillette and Taqqu (2010) obtained the following asymptotic result for $U_{{\alpha_{1},\alpha_{2}}}(t)$:
\begin{equation}\label{Tau}
U_{{\alpha_{1},\alpha_{2}}}(t) \sim \begin{cases}
\dfrac{t^{\alpha_{1}}}{C_{1} \Gamma\left(\alpha_{1}+1\right)}, & {t \rightarrow 0},\\ \dfrac{t^{\alpha_{2}}}{C_{2} \Gamma\left(\alpha_{2}+1\right)}, & {t \rightarrow \infty}.
\end{cases}
\end{equation}
For fixed $s\ge0$ and large $t$, the variance  and covariance of the inverse mixed stable subordinator have the following limiting behaviour (see Kataria and Khandakar (2019)):
 \begin{align}
 \operatorname{Var}(Y_{\alpha_{1},\alpha_{2}}(t))&\sim\frac{t^{2\alpha_{2}}}{C_{2}^{2}}\left(\frac{2}{\Gamma(2\alpha_{2}+1)}-\frac{1}{\left(\Gamma(\alpha_{2}+1)\right)^{2}}\right),\label{3.5}\\
\operatorname{Cov}\left(Y_{\alpha_{1},\alpha_{2}}(s), Y_{\alpha_{1},\alpha_{2}}(t)\right)&\sim\frac{s^{2\alpha_{1}}}{C_{1}^{2}}E_{\alpha_{1}-\alpha_{2}, 2\alpha_{1}+1}^{2}\left(-C_{2}s^{\alpha_{1}-\alpha_{2}}/C_{1}\right).\label{co}
\end{align}
\subsection{The LRD and SRD properties}
The long-range dependence and short-range dependence properties for a non-stationary stochastic process $\{X(t)\}_{t\geq0}$  are defined as follows (see D'Ovidio and Nane (2014), Kumar {\it et al.} (2019)):
\begin{definition}
Let $s>0$ be fixed and $t>s$. If the correlation function of a stochastic process $\{X(t)\}_{t\ge0}$ has the following asymptotic behaviour as $t\rightarrow\infty$: 
\begin{equation}\label{lrd}
\operatorname{Corr}(X(s),X(t))\sim d(s)t^{-\nu},
\end{equation}
where the constant $d(s)$ is  independent of $t$, then $\{X(t)\}_{t\ge0}$ is said to exhibits the LRD property if $\nu\in(0,1)$ and the SRD property if $\nu\in(1,2)$.
\end{definition}
\section{Compound Mixed Fractional Poisson Process}
Beghin (2012) introduced and studied a fractional version of the Poisson process, namely, the MFPP which is based on the inverse mixed stable subordinator. Here, we introduce a compound version of the MFPP that will be used in the next section to construct an insurance risk model. For completeness, we first give brief details on the MFPP.

Consider a time changed Poisson process, namely, the  mixed fractional non-homogeneous Poisson process (MFNPP) defined as (see Aletti {\it et al.} (2018)) 
\begin{equation*}
N_{\Lambda}^{\alpha_{1}, \alpha_{2}}=\left\{N_{\Lambda}^{\alpha_{1}, \alpha_{2}}(t)\right\}_{t \geq 0}=\left\{N\left(\Lambda\left(Y_{\alpha_{1}, \alpha_{2}}(t)\right)\right)\right\}_{t \geq 0}.
\end{equation*}
Here, $0<\alpha_{2}<\alpha_{1}<1$ and $\{N(t)\}_{t\geq0}$ is the Poisson process with intensity 1 which is independent of the inverse mixed stable subordinator $\{Y_{\alpha_{1},\alpha_{2}}(t)\}_{t\ge0}$. Also, $\Lambda:\mathbb{R}^{+}\to\mathbb{R}^{+}$ is a right-continuous non-decreasing function with $\Lambda(0)=0, \ \Lambda(\infty)=\infty$  and  $\Lambda(t)-\Lambda(t-) \leq1$. 

For $\Lambda(t)=\lambda t$, $\lambda>0$, the MFNPP reduces to MFPP. Thus,
\begin{equation*}
N^{\alpha_{1}, \alpha_{2}}=\left\{N^{\alpha_{1}, \alpha_{2}}(t)\right\}_{t \geq 0}=\left\{N\left(Y_{\alpha_{1}, \alpha_{2}}(t)\right)\right\}_{t \geq 0},
\end{equation*}
where  $\{N(t)\}_{t\geq0}$ is the Poisson process with intensity $\lambda$.

Beghin (2012) showed that the MFPP is a renewal process, and obtained the Laplace transform of its interarrival time $W$ as
\begin{equation}\label{laplace}
\tilde{f}_{W}^{\alpha_{1},\alpha_{2}}(s)=\dfrac{\lambda}{ C_{1}s^{\alpha_{1}}+C_{2}s^{\alpha_{2}} +\lambda},\quad s>0,
\end{equation}
where $C_{1}\geq0,C_{2}\geq0$ such that $C_{1}+C_{2}=1$.
On using the inverse Laplace transform, the density of the interarrival time is given by
\begin{equation}\label{density}
f_{W}^{\alpha_{1},\alpha_{2}}(t)=\frac{\lambda t^{\alpha_{1}-1}}{C_{1}}\sum_{k=0}^{\infty}\left(-\frac{C_{2}t^{\alpha_{1}-\alpha_{2}}}{C_{1}} \right)^{k}E_{\alpha_{1},\alpha_{1}+(\alpha_{1}-\alpha_{2})k}^{k+1}\left(-\frac{\lambda t^{\alpha_{1}}}{C_{1}}\right).
\end{equation}
The Laplace transform  of the state probabilities $p_{n}^{\alpha_{1},\alpha_{2}}(t)=\mathrm{Pr}\{N^{\alpha_{1},\alpha_{2}}(t)=n\}$  of the MFPP is given by (see Beghin (2012), Eq. (2.7))
\begin{equation}\label{invl}
\tilde{p}_{n}^{\alpha_{1},\alpha_{2}}(s)=\dfrac{\lambda^{n}C_{1}s^{\alpha_{1}-1}}{(C_{1}s^{\alpha_{1}}+C_{2}s^{\alpha_{2}}+\lambda)^{n+1}}+\dfrac{\lambda^{n}C_{2}s^{\alpha_{2}-1}}{(C_{1}s^{\alpha_{1}}+C_{2}s^{\alpha_{2}}+\lambda)^{n+1}}, \quad n\ge0.
\end{equation}
The state probabilities of MFPP satisfy the following system of fractional differential equations (see Beghin (2012), Eq 2.6):
\begin{equation}\label{gov0}
C_{1}\partial^{\alpha_{1}}_{t}p_{n}^{\alpha_{1},\alpha_{2}}(t)+C_{2}\partial^{\alpha_{2}}_{t}p_{n}^{\alpha_{1},\alpha_{2}}(t)=-\lambda\left(p_{n}^{\alpha_{1},\alpha_{2}}(t)-p_{n-1}^{\alpha_{1},\alpha_{2}}(t)\right),\ n\ge0,
\end{equation}
subject to the initial conditions 
\begin{equation*}
p_{0}^{\alpha_{1},\alpha_{2}}(0)=1,\ p_{n}^{\alpha_{1},\alpha_{2}}(0)=0, \ n\ge1,
\end{equation*}
with $p_{-1}^{\alpha_{1},\alpha_{2}}(t)=0$. Here, $\partial^{\alpha}_{t}$ denotes the Caputo fractional derivative defined in (\ref{plm1}). The inversion of the above Laplace transform (\ref{invl}) is tedious. For $n=0$, the explicit form of $p_{0}^{\alpha_{1},\alpha_{2}}(t)$  in terms of three-parameter Mittag-Leffler function is given by (see Beghin (2012), Eq. (2.13))
\begin{align}\label{qp09iijh1}
p_{0}^{\alpha_{1},\alpha_{2}
}(t)&=\sum_{k=0}^{\infty}\left(-\frac{C_{2}t^{\alpha_{1}-\alpha_{2}}}{C_{1}} \right)^{k}E_{\alpha_{1},(\alpha_{1}-\alpha_{2})k+1}^{k+1}\left(-\frac{\lambda t^{\alpha_{1}}}{C_{1}}\right)\nonumber\\
&\ \ \ \ \ \ \ -\sum_{k=0}^{\infty}\left(-\frac{C_{2}t^{\alpha_{1}-\alpha_{2}}}{C_{1}} \right)^{k+1}E_{\alpha_{1},(\alpha_{1}-\alpha_{2})(k+1)+1}^{k+1}\left(-\frac{\lambda t^{\alpha_{1}}}{C_{1}}\right).
\end{align}
Beghin (2012) stated that the explicit expressions for the state probabilities of MFPP can be obtained in terms of convolutions of known distributions. Next, we give the exact form of the state probabilities of MFPP in terms of convolutions of two given functions.
\begin{proposition}
The state probabilities  $p_{n}^{\alpha_{1},\alpha_{2}}(t),\ n\ge0$ of the MFPP is given by
\begin{equation}\label{keyrrd}
p_{n}^{\alpha_{1},\alpha_{2}}(t)=\dfrac{\lambda^n}{C_{1}^{n+1}}\left(C_1f^{*(n+1)}(t)+C_2g^{*(n+1)}(t)\right),
\end{equation}
where
\begin{align*}
f(t)&=t^{\frac{n(\alpha_{1}-1)}{n+1}}\sum_{k=0}^{\infty}\left(-\frac{C_{2}t^{\alpha_{1}-\alpha_{2}}}{C_{1}} \right)^{k}E_{\alpha_{1},(\alpha_{1}-\alpha_{2})k+\frac{n\alpha_{1}+1}{n+1}}^{k+1}\left(-\frac{\lambda t^{\alpha_{1}}}{C_{1}}\right),\\
g(t)&=t^{\frac{n(\alpha_{1}-1)+\alpha_{1}-\alpha_{2}}{n+1}}\sum_{k=0}^{\infty}\left(-\frac{C_{2}t^{\alpha_{1}-\alpha_{2}}}{C_{1}} \right)^{k}E_{\alpha_{1},(\alpha_{1}-\alpha_{2})k+\frac{n\alpha_{1}+\alpha_{1}-\alpha_{2}+1}{n+1}}^{k+1}\left(-\frac{\lambda t^{\alpha_{1}}}{C_{1}}\right),
\end{align*}
and $f^{*(n+1)}$ and $g^{*(n+1)}$ denote $(n+1)$-fold convolution of $f$ and  $g$ respectively.
\end{proposition}
\begin{proof}
On taking the inverse Laplace transform in (\ref{invl}), we get
\begin{equation*}
p_{n}^{\alpha_{1},\alpha_{2}}(t)=\frac{\lambda^n}{C_1^n}\mathscr{L}^{-1}\left(\left(\dfrac{s^{\frac{\alpha_{1}-1}{n+1}}}{s^{\alpha_{1}}+\frac{C_{2}}{C_{1}}s^{\alpha_{2}}+\frac{\lambda}{C_{1}}}\right)^{n+1};t\right) +\frac{\lambda^nC_2}{C_1^{n+1}}\mathscr{L}^{-1}\left(\left(\dfrac{s^{\frac{\alpha_{2}-1}{n+1}}}{s^{\alpha_{1}}+\frac{C_{2}}{C_{1}}s^{\alpha_{2}}+\frac{\lambda}{C_{1}}}\right)^{n+1};t\right).       
\end{equation*}
On using the following result due to  Haubold \textit{et al.} (2011), Eq. (17.6):
\begin{equation*}
\mathscr{L}^{-1}\left(\dfrac{s^{\rho-1}}{s^{\alpha}+as^{\beta}+b};t\right)=t^{\alpha-\rho}\sum_{k=0}^{\infty}(-a)^{k}t^{(\alpha-\beta)k}E_{\alpha,\alpha+(\alpha-\beta)k-\rho+1}^{k+1}\left(-bt^{\alpha}\right),
\end{equation*}
where $\alpha>0, \beta>0, \rho>0$ and $|as^{\beta}/(s^{\alpha}+b)|<1$,
we get
\begin{equation*}
p_{n}^{\alpha_{1},\alpha_{2}}(t)=\left(\dfrac{\lambda}{C_{1}}\right)^{n}f^{*(n+1)}(t)+\left(\dfrac{\lambda}{C_{1}}\right)^{n}\dfrac{C_{2}}{C_{1}}g^{*(n+1)}(t).
\end{equation*}
This completes the proof.
\end{proof}
On putting $n=0$ in (\ref{keyrrd}) we obtain (\ref{qp09iijh1}). The probability generating function (pgf) of MFPP is given by Beghin (2012) as 
\begin{align}\label{qmmmqa1}
\mathbb{E}(z^{N^{\alpha_{1}, \alpha_{2}}(t)})&=\sum_{k=0}^{\infty}\left(-\frac{C_{2}t^{\alpha_{1}-\alpha_{2}}}{C_{1}} \right)^{k}E_{\alpha_{1},(\alpha_{1}-\alpha_{2})k+1}^{k+1}\left(-\frac{\lambda (1-z)t^{\alpha_{1}}}{C_{1}}\right)\nonumber\\
&\ \ \  +\sum_{k=0}^{\infty}\left(-\frac{C_{2}t^{\alpha_{1}-\alpha_{2}}}{C_{1}}\right)^{k+1}E_{\alpha_{1},(\alpha_{1}-\alpha_{2})(k+1)+1}^{k+1}\left(-\frac{\lambda (1-z)t^{\alpha_{1}}}{C_{1}}\right).
\end{align}
\subsection{The Compound MFPP}
We define a compound version of the MFPP denoted by $\{C^{\alpha_{1},\alpha_{2}}(t)\}_{t\ge0}$ as
\begin{equation}\label{qmlk11}
	C^{\alpha_{1},\alpha_{2}}(t)\coloneqq\sum_{i=1}^{N^{\alpha_{1},\alpha_{2}}(t)}X_{i},
\end{equation}
where $\{X_i\}_{i\geq1}$ is a sequence of positive integer valued iid random variables with finite mean which is independent of the MFPP $\{N^{\alpha_{1},\alpha_{2}}(t)\}_{t\ge0}$. Thus,
\begin{equation*}
C^{\alpha_{1},\alpha_{2}}(t)=\sum_{i=1}^{N(Y_{\alpha_1,\alpha_2}(t))}X_{i},
\end{equation*}
where the inverse mixed stable subordinator $\{Y_{\alpha_{1},\alpha_{2}}(t)\}_{t\ge0}$ is independent of the Poisson process $\{N(t)\}_{t\ge0}$ with intensity parameter $\lambda$. 

Next, we show that the MFPP and compound MFPP exhibits overdispersion. A stochastic process $\{X(t)\}_{t>0}$ is said to have overdispersion if  $\operatorname{Var}(X(t))-\mathbb{E}(X(t))>0$ for all $t>0$. Beghin and Macci (2014) showed that the TFPP and compound TFPP are overdispersed processes. The mean and variance of MFPP are given by (see Aletti {\it et al.} (2018))
\begin{align}
\mathbb{E}\left(N^{\alpha_{1}, \alpha_{2}}(t)\right)
&=\lambda U_{{\alpha_{1},\alpha_{2}}}(t),\label{E}\\ \operatorname{Var}\left(N^{\alpha_{1}, \alpha_{2}}(t)\right)
&=\lambda U_{\alpha_{1},\alpha_{2}}(t)+\lambda^{2}\operatorname {Var} \left(Y_{\alpha_{1}, \alpha_{2}}(t)\right),\label{V}
\end{align}
where $U_{{\alpha_{1},\alpha_{2}}}(t)$ is given in (\ref{7}). Therefore,
\begin{equation*}
\operatorname{Var}\left(N^{\alpha_{1}, \alpha_{2}}(t)\right)-\mathbb{E}\left(N^{\alpha_{1}, \alpha_{2}}(t)\right)=\lambda^{2}\operatorname {Var} \left(Y_{\alpha_{1}, \alpha_{2}}(t)\right)>0,
\end{equation*}
for all $t>0$. This shows that the MFPP exhibits overdispersion.

Using a well known result (see Mikosh (2009)), the mean and variance of $\{C^{\alpha_{1},\alpha_{2}}(t)\}_{t\ge0}$ are obtained as
\begin{equation*}
	\mathbb{E}\left(C^{\alpha_{1},\alpha_{2}}(t)\right)=\mathbb{E}\left(N^{\alpha_{1}, \alpha_{2}}(t)\right)\mathbb{E}(X_{1})=\lambda U_{{\alpha_{1},\alpha_{2}}}(t)\mathbb{E}(X_{1}),
\end{equation*}	
and 
\begin{align*}
	\operatorname{Var}\left(C^{\alpha_{1},\alpha_{2}}(t)\right)&=\mathbb{E}\left(N^{\alpha_{1}, \alpha_{2}}(t)\right)\operatorname{Var}(X_{1})+\operatorname{Var}\left(N^{\alpha_{1}, \alpha_{2}}(t)\right)(\mathbb{E}(X_{1}))^{2}\\
	&=\lambda U_{{\alpha_{1},\alpha_{2}}}(t)\operatorname{Var}(X_{1})+\left(\lambda U_{\alpha_{1},\alpha_{2}}(t)+\lambda^{2}\operatorname {Var} \left(Y_{\alpha_{1}, \alpha_{2}}(t)\right)\right)(\mathbb{E}(X_{1}))^{2}\\
	&=\lambda U_{{\alpha_{1},\alpha_{2}}}(t)\mathbb{E}(X_{1}^2)+\lambda^{2}\operatorname {Var} \left(Y_{\alpha_{1}, \alpha_{2}}(t)\right)(\mathbb{E}(X_{1}))^{2}.
\end{align*}
Since $\mathbb{E}(X_{1}^{2})-\mathbb{E}(X_{1})\ge 0$ as $\mathrm{Pr}\{X_{1}\ge 1\}=1$, we have
\begin{equation*}
\operatorname{Var}\left(C^{\alpha_{1},\alpha_{2}}(t)\right)-\mathbb{E}\left(C^{\alpha_{1},\alpha_{2}}(t)\right)>0.
\end{equation*}
This establishes that compound MFPP exhibits overdispersion.

The following result gives the governing fractional differential equations of the state probabilities of compound MFPP. 
\begin{proposition}
The state probabilities $q^{\alpha_{1},\alpha_{2}}_n(t)=\mathrm{Pr}\{C^{\alpha_{1},\alpha_{2}}(t)=n\}$ of compound MFPP satisfy the following fractional differential equation
\begin{align*}
C_1\partial_t^{\alpha_1}q^{\alpha_{1},\alpha_{2}}_0(t)+C_2\partial_t^{\alpha_2}q^{\alpha_{1},\alpha_{2}}_0(t)&=-\lambda q^{\alpha_{1},\alpha_{2}}_0(t),\\
C_1\partial_t^{\alpha_1}q^{\alpha_{1},\alpha_{2}}_n(t)+C_2\partial_t^{\alpha_2}q^{\alpha_{1},\alpha_{2}}_n(t)&=-\lambda q^{\alpha_{1},\alpha_{2}}_n(t)+\lambda\sum_{i=1}^{n}\mathrm{Pr}\{X_1=i\}q^{\alpha_{1},\alpha_{2}}_{n-i}(t),\ \ n\geq1,
	\end{align*}
where $C_{1}\geq0,C_{2}\geq0$ such that $C_{1}+C_{2}=1$.
\end{proposition}
\begin{proof}
Let us denote $r_{i}=\mathrm{Pr}\{X_{1}=i\}$ for all $i\ge 1$ and $r_{n}(k)=\mathrm{Pr}\{X_{1}+X_{2}+\cdots +X_{k}=n\}$. Note that $r_{i}=\mathrm{Pr}\{X_{k}=i\}$ for all $k\geq1$. Thus,
	\begin{equation}\label{pmf}
	q^{\alpha_{1},\alpha_{2}}_n(t)=\begin{cases}
	p_{0}^{\alpha_{1},\alpha_{2}}(t)&\mathrm{if}\  n=0,\\
	\displaystyle\sum_{k=1}^{n}r_{n}(k)p_{k}^{\alpha_{1},\alpha_{2}}(t)& \mathrm{if}\  n\ge1.
	\end{cases}
	\end{equation}
From the above equation, we have
\begin{align*}
C_1\partial_t^{\alpha_1}q^{\alpha_{1},\alpha_{2}}_0(t)+C_2\partial_t^{\alpha_2}q^{\alpha_{1},\alpha_{2}}_0(t)&=C_1\partial_t^{\alpha_1}p^{\alpha_{1},\alpha_{2}}_0(t)+C_2\partial_t^{\alpha_2}p^{\alpha_{1},\alpha_{2}}_0(t)\\
&=-\lambda p^{\alpha_{1},\alpha_{2}}_0(t),\ \ \mathrm{(using\ (\ref{gov0}))}\\
&=-\lambda q^{\alpha_{1},\alpha_{2}}_0(t).
\end{align*}
Thus, the result holds true for $n=0$. For $n\geq1$, we get
\begin{align*}
C_{1}\partial^{\alpha_{1}}_{t}q_{n}^{\alpha_{1},\alpha_{2}}(t)&=\displaystyle\sum_{k=1}^{n}r_{n}(k)C_{1}\partial^{\alpha_{1}}_{t}p_{k}^{\alpha_{1},\alpha_{2}}(t),\\
C_{2}\partial^{\alpha_{2}}_{t}q_{n}^{\alpha_{1},\alpha_{2}}(t)&=\displaystyle\sum_{k=1}^{n}r_{n}(k)C_{2}\partial^{\alpha_{2}}_{t}p_{k}^{\alpha_{1},\alpha_{2}}(t).
\end{align*}	
Adding the above equations and using (\ref{gov0}), we get
\begin{align}\label{+}
C_{1}\partial^{\alpha_{1}}_{t}q_{n}^{\alpha_{1},\alpha_{2}}(t)+C_{2}\partial^{\alpha_{2}}_{t}q_{n}^{\alpha_{1},\alpha_{2}}(t)&=-\lambda\sum_{k=1}^{n}r_{n}(k)p_{k}^{\alpha_{1},\alpha_{2}}(t)+\lambda\sum_{k=1}^{n}r_{n}(k)p_{k-1}^{\alpha_{1},\alpha_{2}}(t)\nonumber\\
&=-\lambda q^{\alpha_{1},\alpha_{2}}_n(t)+\lambda\sum_{k=1}^{n}\left(\sum_{i=1}^{n}r_{n-i}(k-1)r_{i}\right)p_{k-1}^{\alpha_{1},\alpha_{2}}(t)\nonumber\\
&=-\lambda q^{\alpha_{1},\alpha_{2}}_n(t)+\lambda\sum_{i=1}^{n}r_{i}\sum_{k=1}^{n}r_{n-i}(k-1)p_{k-1}^{\alpha_{1},\alpha_{2}}(t).
\end{align}
Note that $r_{k}(0)=0$, $r_{0}(k)=0$ for all $k\ge1$, $r_{0}(0)=1$, and $r_{n}(k)=0$ for $n<k$. Also,
\begin{equation}\label{+-}
q^{\alpha_{1},\alpha_{2}}_n(t)=\displaystyle\sum_{k=1}^{m}r_{n}(k)p_{k}^{\alpha_{1},\alpha_{2}}(t)\ \  \mathrm{for\  all\  integers}\  m\ge n.
\end{equation}
Now,
\begin{align}\label{-}
\sum_{i=1}^{n}r_{i}\sum_{k=1}^{n}r_{n-i}(k-1)p_{k-1}^{\alpha_{1},\alpha_{2}}(t)&=\sum_{i=1}^{n-1}r_{i}\sum_{k=1}^{n}r_{n-i}(k-1)p_{k-1}^{\alpha_{1},\alpha_{2}}(t)+r_{n}\sum_{k=1}^{n}r_{0}(k-1)p_{k-1}^{\alpha_{1},\alpha_{2}}(t)\nonumber\\
&=\sum_{i=1}^{n-1}r_{i}\sum_{k=2}^{n}r_{n-i}(k-1)p_{k-1}^{\alpha_{1},\alpha_{2}}(t)+r_{n}p_{0}^{\alpha_{1},\alpha_{2}}(t)\nonumber\\
&=\sum_{i=1}^{n-1}r_{i}\sum_{j=1}^{n-1}r_{n-i}(j)p_{j}^{\alpha_{1},\alpha_{2}}(t)+r_{n}q_{0}^{\alpha_{1},\alpha_{2}}(t),\ \  \mathrm{(using\ (\ref{pmf}))}\nonumber\\
&=\sum_{i=1}^{n-1}r_{i}q_{n-i}^{\alpha_{1},\alpha_{2}}(t)+r_{n}q_{0}^{\alpha_{1},\alpha_{2}}(t),\ \  \mathrm{(using\ (\ref{+-}))}\nonumber\\
&=\sum_{i=1}^{n}r_{i}q_{n-i}^{\alpha_{1},\alpha_{2}}(t).
\end{align}
Substituting (\ref{-}) in (\ref{+}), we get the result.
\end{proof}
\section{Mixed Fractional Risk Process}
In this section,  we introduce a fractional version of the risk process given in (\ref{9}) by considering the inverse mixed stable subordinator in the premium received by the insurance company till time $t$. The MFPP is used to model the total claim received. 

Consider a fractional risk process defined as
\begin{equation}\label{1}
R_{\Lambda}^{\alpha_{1},\alpha_{2}}(t)\coloneqq u+\mu(1+\rho)\Lambda(Y_{\alpha_{1},\alpha_{2}}(t))-\sum_{i=1}^{N_{\Lambda}^{\alpha_{1},\alpha_{2}}(t)}X_{i}, \ \ t\ge0,
\end{equation}
 where the initial capital is $u>0$, the safety loading parameter is $\rho\ge0$ and  the iid random variables $X_{i}$'s have finite mean  $\mu>0$. Note that the fractional risk process $\{R_{\Lambda}^{\alpha_{1},\alpha_{2}}(t)\}_{t\ge0}$ is based on the mixed inverse stable subordinator and the MFNPP. It is assumed that the MFNPP $\{N_{\Lambda}^{\alpha_{1},\alpha_{2}}(t)\}_{t\ge0}$ is independent of the claims $\left\{X_{i}\right\}_{i \geq 1}$. This process is of interest because of its  martingale characterization.
 
 In the following theorem, we show that the fractional risk process defined in (\ref{1}) satisfies the net profit condition for $\rho>0$.

\begin{theorem}\label{6}
The fractional risk process  $\{R_{\Lambda}^{\alpha_{1},\alpha_{2}}(t)\}_{t\ge0}$ is a martingale (submartingale/supermartingale) for $\rho=0$ ($\rho>0/\rho<0$) with respect to the filtration
\begin{equation*}
\mathcal{F}_{t}=\sigma\left(N^{\alpha_{1},\alpha_{2}}_\Lambda(s), s \leq t\right) \vee \sigma\left(Y_{\alpha_{1},\alpha_{2}}(s), s \geq 0\right).
\end{equation*}
\end{theorem}
\begin{proof}
Note that the fractional risk process  $\{R_{\Lambda}^{\alpha_{1},\alpha_{2}}(t)\}_{t\ge0}$ is adapted to the given filtration $\mathcal{F}_{t}$. For $s\le t$, we have
\begin{align*}
&\mathbb{E}\left(R_{\Lambda}^{\alpha_{1},\alpha_{2}}(t)| \mathcal{F}_{s}\right)\\
&=\mathbb{E}\left(u+ \mu(1+\rho)\Lambda(Y_{\alpha_{1},\alpha_{2}}(t))-\sum_{i=1}^{N_{\Lambda}^{\alpha_{1},\alpha_{2}}(t)}X_{i}\Bigg|\mathcal{F}_{s}\right)\\
&=R_{\Lambda}^{\alpha_{1},\alpha_{2}}(s)+\mathbb{E}\left( \mu(1+\rho)\big(\Lambda(Y_{\alpha_{1},\alpha_{2}}(t))-\Lambda(Y_{\alpha_{1},\alpha_{2}}(s))\big)-\sum_{i=N_\Lambda^{\alpha_{1},\alpha_{2}}(s)+1}^{N_{\Lambda}^{\alpha_{1},\alpha_{2}}(t)}X_{i}\Bigg|\mathcal{F}_{s}\right)\\
&=R_{\Lambda}^{\alpha_{1},\alpha_{2}}(s)+
\mathbb{E} \left(\mu(1+\rho)\big(\Lambda(Y_{\alpha_{1},\alpha_{2}}(t))-\Lambda(Y_{\alpha_{1},\alpha_{2}}(s))\big)\big|\mathcal{F}_{s}\right)-\mathbb{E}\left(\mathbb{E}\left(\sum_{i=N_{\Lambda}^{\alpha_{1},\alpha_{2}}(s)+1}^{N_{\Lambda}^{\alpha_{1},\alpha_{2}}(t)} X_{i} \Bigg| \mathcal{F}_{t}\right) \Bigg| \mathcal{F}_{s}\right)\\
&=R_{\Lambda}^{\alpha_{1},\alpha_{2}}(s)+\mathbb{E}\left( \mu(1+\rho)\big(\Lambda(Y_{\alpha_{1},\alpha_{2}}(t))-\Lambda(Y_{\alpha_{1},\alpha_{2}}(s))\big)\big|\mathcal{F}_{s}\right)-\mathbb{E}\left(\left(N_{\Lambda}^{\alpha_{1},\alpha_{2}}(t)-N_{\Lambda}^{\alpha_{1},\alpha_{2}}(s)\right) \mu \big| \mathcal{F}_{s}\right)\\
&=R_{\Lambda}^{\alpha_{1},\alpha_{2}}(s)+\mu\rho\mathbb{E}\left( \big(\Lambda(Y_{\alpha_{1},\alpha_{2}}(t))-\Lambda(Y_{\alpha_{1},\alpha_{2}}(s))\big) \big| \mathcal{F}_{s}\right)\\
&\ \ -\mu \mathbb{E}\big(\left(N_{\Lambda}^{\alpha_{1},\alpha_{2}}(t)-\Lambda Y_{\alpha_{1},\alpha_{2}}(t)\right)-\left(N_{\Lambda}^{\alpha_{1},\alpha_{2}}(s)-\Lambda Y_{\alpha_{1},\alpha_{2}}(s)\right) \big| \mathcal{F}_{s}\big)\\
&=R_{\Lambda}^{\alpha_{1},\alpha_{2}}(s)+\mu\rho\mathbb{E}\left( \big(\Lambda(Y_{\alpha_{1},\alpha_{2}}(t))-\Lambda(Y_{\alpha_{1},\alpha_{2}}(s))\big)\big| \mathcal{F}_{s}\right), 
\end{align*}
where in the last step we have used that the process $\left\{N_{\Lambda}^{\alpha_{1},\alpha_{2}}(t)-\Lambda( Y_{\alpha_{1},\alpha_{2}}(t))\right\}_{t\ge0}$  is a martingale with respect to the given filtration $\mathcal{F}_{t}$ (see Aletti {\it et al.} (2018)). Therefore, the fractional risk process $\{R_{\Lambda}^{\alpha_{1},\alpha_{2}}(t)\}_{t\ge0}$ is a martingale if $\rho=0$. Since $\{\Lambda(Y_{\alpha_{1},\alpha_{2}}(t))\}_{t\ge0}$ is an increasing process, it follows that $\{R_{\Lambda}^{\alpha_{1},\alpha_{2}}(t)\}_{t\ge0}$ is a submartingale (supermartingale) if $\rho>0 \ (\rho<0)$.
\end{proof}

For $\Lambda(t)=\lambda t$, $\lambda>0$, the fractional risk process defined in (\ref{1}) reduces to a special case which we call as the mixed fractional risk process (MFRP). It is defined as 
\begin{equation}\label{*}
R^{\alpha_{1},\alpha_{2}}(t)=u+\mu(1+\rho)\lambda Y_{\alpha_{1},\alpha_{2}}(t)-\sum_{i=1}^{N^{\alpha_{1},\alpha_{2}}(t)}X_{i}, \ \ t\ge0.
\end{equation}
 The claim numbers in MFRP are modelled using MFPP. As we have shown that the fractional risk process $\{R_{\Lambda}^{\alpha_{1},\alpha_{2}}(t)\}_{t\ge0}$ is a martingale, it follows that MFRP  $\{R^{\alpha_{1},\alpha_{2}}(t)\}_{t\ge0}$ is also a martingale (submartingale/supermartingale) for $\rho=0$ ($\rho>0$/$\rho<0$) with respect to the filtration 	$\mathcal{F}_{t}=\sigma\left(N^{\alpha_{1},\alpha_{2}}(s), s \leq t\right) \vee \sigma\left(Y_{\alpha_{1},\alpha_{2}}(s), s \geq 0\right)$.
 
 The net or equivalence principle (see Mikosch (2009), Section 3.1.3) remains unaffected by the time change imposed by $Y_{\alpha_{1},\alpha_{2}}$ in (\ref{*}) as $\{(X_{1}+X_{2}+\dots+X_{N^{\alpha_{1}, \alpha_{2}}(t)})-\mu\lambda Y_{\alpha_{1},\alpha_{2}}(t)\}_{t\ge0}$ is a martingale. Note that $\mathbb{E}\left(N^{\alpha_{1}, \alpha_{2}}(t)\right)=\lambda\mathbb{E}\left(Y_{\alpha_{1},\alpha_{2}}(t)\right)$. For $C_1=0$, the inverse mixed stable subordinator reduces to inverse stable subordinator and thus MFPP reduces to TFPP. Therefore, the fractional risk process $\{R^{\alpha}(t)\}_{t\ge0}$ introduced and studied by Kumar {\it et al.} (2019) becomes a special case of MFRP. Also, the expectation of the MFRP is given by
 \begin{equation*}
\mathbb{E} \left(R^{\alpha_{1},\alpha_{2}}(t)\right)=u+\mu\rho\lambda U_{{\alpha_{1},\alpha_{2}}}(t) ,
 \end{equation*}
where $U_{{\alpha_{1},\alpha_{2}}}(t)$ is given by (\ref{7}).
\subsubsection{The LRD property of MFRP}Next, we establish the LRD property of MFRP. For that purpose, we require the variance and covariance of MFRP which are evaluated as follows.
\begin{proposition}\label{covr}
The covariance of MFRP is given by
	\begin{equation*}
	{\operatorname{Cov}\left(R^{\alpha_{1},\alpha_{2}}(s), R^{\alpha_{1},\alpha_{2}}(t)\right)=\mu^{2} \lambda^{2} \rho^{2} \operatorname{Cov}\left(Y_{\alpha_{1},\alpha_{2}}(s), Y_{\alpha_{1},\alpha_{2}}(t)\right)+\mathbb{E} \left(X_{1}^{2}\right)\mathbb{E}\left(N^{\alpha_{1},\alpha_{2}}(s)\right)},
	\end{equation*}
	where $0\leq s\leq t$. Hence, its variance is given by
	\begin{equation*}
	\operatorname{Var}\left(R^{\alpha_{1},\alpha_{2}}(t)\right)=\mu^{2} \lambda^{2} \rho^{2} \operatorname{Var}\left(Y_{\alpha_{1},\alpha_{2}}(t)\right)+\mathbb{E}\left( X_{1}^{2}\right)\mathbb{E}\left(N^{\alpha_{1},\alpha_{2}}(t)\right).
	\end{equation*}
\end{proposition}
\begin{proof}
We recall that $\mu=\mathbb{E}(X_1)$. From (\ref{*}), we have for $s\leq t$
\begin{align}\label{R}
\operatorname{Cov}\left(R^{\alpha_{1},\alpha_{2}}(s), R^{\alpha_{1},\alpha_{2}}(t)\right)&=I(s,t)-\lambda(1+\rho)\mathbb{E}(X_1)(J(s,t)+K(s,t))\nonumber\\
&\ \ +\lambda^{2} (1+\rho)^{2} (\mathbb{E}(X_1))^{2} \operatorname{Cov}\left(Y_{\alpha_{1},\alpha_{2}}(s), Y_{\alpha_{1},\alpha_{2}}(t)\right),
\end{align}
where 
\begin{align*}
I(s,t)&=\operatorname{Cov}\left(\sum_{j=1}^{N^{\alpha_{1},\alpha_{2}}(s)} X_{j}, \sum_{i=1}^{N^{\alpha_{1},\alpha_{2}}(t)} X_{i}\right),\nonumber\\
J(s,t)&=\operatorname{Cov}\left(Y_{\alpha_{1},\alpha_{2}}(s), \sum_{i=1}^{N^{\alpha_{1},\alpha_{2}}(t)} X_{i}\right),\ \ K(s,t)=\operatorname{Cov}\left(Y_{\alpha_{1},\alpha_{2}}(t), \sum_{j=1}^{N^{\alpha_{1},\alpha_{2}}(s)} X_{j}\right).
\end{align*}
Now, 
\begin{align}\label{asde1v}
I(s,t)&=-\left(\mathbb{E} X_{1}\right)^{2} \mathbb{E}\left(N^{\alpha_{1},\alpha_{2}}(s)\right)\mathbb{E}\left(N^{\alpha_{1},\alpha_{2}}(t)\right)\nonumber\\
&\ \ \ \ +\mathbb{E}\left(\sum_{i=1}^{\infty} \sum_{j=1}^{\infty} X_{i} X_{j} I\left\{N^{\alpha_{1},\alpha_{2}}(s) \geq j, N^{\alpha_{1},\alpha_{2}}(t) \geq i\right\}\right)\nonumber\\
&=-\left(\mathbb{E} X_{1}\right)^{2} \mathbb{E}\left(N^{\alpha_{1},\alpha_{2}}(s)\right)\mathbb{E}\left(N^{\alpha_{1},\alpha_{2}}(t)\right)+\mathbb{E}\left(\sum_{j=1}^{\infty} X_{j}^{2} I\left\{N^{\alpha_{1},\alpha_{2}}(s) \geq j\right\}\right)\nonumber\\ 
&\ \ \ \ +\mathbb{E}\left(\underset{i \neq j}{\sum\sum}X_{i} X_{j} I\left\{N^{\alpha_{1},\alpha_{2}}(s) \geq j, N^{\alpha_{1},\alpha_{2}}(t) \geq i\right\}\right)\nonumber\\
&=-\left(\mathbb{E} X_{1}\right)^{2} \mathbb{E}\left(N^{\alpha_{1},\alpha_{2}}(s)\right)\mathbb{E}\left(N^{\alpha_{1},\alpha_{2}}(t)\right)+\mathbb{E}(X_{j}^{2})\sum_{j=1}^{\infty}  \mathrm{Pr}\{N^{\alpha_{1},\alpha_{2}}(s) \geq j\}\nonumber\\
&\ \ \ \ +(\mathbb{E}X_{1})^{2}\left(\sum_{i=1}^{\infty} \sum_{j=1}^{\infty}\mathrm{Pr}\{N^{\alpha_{1},\alpha_{2}}(s) \geq j, N^{\alpha_{1},\alpha_{2}}(t) \geq i\}-\sum_{j=1}^{\infty}  \mathrm{Pr}\{N^{\alpha_{1},\alpha_{2}}(s) \geq j\}\right)\nonumber\\
&=\operatorname{Var}(X_{1})\mathbb{E}\big(N^{\alpha_{1},\alpha_{2}}(s)\big) + (\mathbb{E}(X_1))^{2}\operatorname{Cov}\left(N^{\alpha_{1},\alpha_{2}}(s) ,N^{\alpha_{1},\alpha_{2}}(t)\right).
\end{align}
The covariance of MFPP is obtained by Aletti {\it et al.} (2018) which is given by
\begin{equation}\label{qazq1}
\operatorname{Cov}\left(N^{\alpha_{1}, \alpha_{2}}(s), N^{\alpha_{1}, \alpha_{2}}(t)\right)
=\mathbb{E}\big(N^{\alpha_{1},\alpha_{2}}(s)\big)+\lambda^{2}\operatorname{Cov}\left(Y_{\alpha_{1},\alpha_{2}}(s), Y_{\alpha_{1},\alpha_{2}}(t)\right).
\end{equation}
Substituting (\ref{qazq1}) in (\ref{asde1v}), we get
\begin{equation}\label{covxx}
I(s,t)=\mathbb{E}(X_{1}^{2})\mathbb{E}\big(N^{\alpha_{1},\alpha_{2}}(s)\big)+\lambda^{2}(\mathbb{E}(X_1))^{2}\operatorname{Cov}\left(Y_{\alpha_{1},\alpha_{2}}(s), Y_{\alpha_{1},\alpha_{2}}(t)\right).
\end{equation}
Also,
\begin{align}\label{Covyx}
J(s,t)&= \mathbb{E}\left(Y_{\alpha_{1},\alpha_{2}}(s) \sum_{i=1}^{N^{\alpha_{1},\alpha_{2}}(t)} X_{i}\right)-\mathbb{E}\left(Y_{\alpha_{1},\alpha_{2}}(s)\right)\mathbb{E}\left(\sum_{i=1}^{N^{\alpha_{1},\alpha_{2}}(t)} X_{i}\right)\nonumber\\
&=\mathbb{E}\left(\mathbb{E}\left(Y_{\alpha_{1},\alpha_{2}}(s) \sum_{i=1}^{N^{\alpha_{1},\alpha_{2}}(t)} X_{i}\Bigg| Y_{\alpha_{1},\alpha_{2}}(s),N^{\alpha_{1},\alpha_{2}}(t)\right)\right)-\mathbb{E}(X_1)\mathbb{E}\big(N^{\alpha_{1},\alpha_{2}}(t)\big)\mathbb{E}(Y_{\alpha_{1},\alpha_{2}}(s))\nonumber\\
&=\mathbb{E}\left(Y_{\alpha_{1},\alpha_{2}}(s)\mathbb{E}\left(\sum_{i=1}^{N^{\alpha_{1},\alpha_{2}}(t)} X_{i}\Bigg| Y_{\alpha_{1},\alpha_{2}}(s),N^{\alpha_{1},\alpha_{2}}(t)\right)\right)-\mathbb{E}(X_1)\mathbb{E}\big(N^{\alpha_{1},\alpha_{2}}(t)\big)\mathbb{E}(Y_{\alpha_{1},\alpha_{2}}(s))\nonumber\\
&=\mathbb{E}(X_1)\mathbb{E}\big(Y_{\alpha_{1},\alpha_{2}}(s)N^{\alpha_{1},\alpha_{2}}(t)\big)-\lambda\mathbb{E}(X_1)\mathbb{E}(Y_{\alpha_{1},\alpha_{2}}(s))\mathbb{E}(Y_{\alpha_{1},\alpha_{2}}(t))\nonumber\\
&=\mathbb{E}(X_1)\mathbb{E}\left(\mathbb{E}\big(Y_{\alpha_{1},\alpha_{2}}(s)N^{\alpha_{1},\alpha_{2}}(t)\big|Y_{\alpha_{1},\alpha_{2}}(l):0\le l\le t\big)\right)-\lambda\mathbb{E}(X_1)\mathbb{E}(Y_{\alpha_{1},\alpha_{2}}(s))\mathbb{E}(Y_{\alpha_{1},\alpha_{2}}(t))\nonumber\\
&=\mathbb{E}(X_1)\mathbb{E}\left(Y_{\alpha_{1},\alpha_{2}}(s)\mathbb{E}\big(N^{\alpha_{1},\alpha_{2}}(t)\big|Y_{\alpha_{1},\alpha_{2}}(l):0\le l\le t\big)\right)-\lambda\mathbb{E}(X_1)\mathbb{E}(Y_{\alpha_{1},\alpha_{2}}(s))\mathbb{E}(Y_{\alpha_{1},\alpha_{2}}(t))\nonumber\\
&=\lambda\mathbb{E}(X_1)\mathbb{E}(Y_{\alpha_{1},\alpha_{2}}(s)Y_{\alpha_{1},\alpha_{2}}(t))-\lambda\mathbb{E}(X_1)\mathbb{E}(Y_{\alpha_{1},\alpha_{2}}(s))\mathbb{E}(Y_{\alpha_{1},\alpha_{2}}(t))\nonumber\\
&=\lambda\mathbb{E}(X_1)\operatorname{Cov}\left(Y_{\alpha_{1},\alpha_{2}}(s), Y_{\alpha_{1},\alpha_{2}}(t)\right).
\end{align}
Similarly,
\begin{equation}\label{covyxj}
K(s,t)=\lambda\mathbb{E}(X_1)\operatorname{Cov}\left(Y_{\alpha_{1},\alpha_{2}}(s), Y_{\alpha_{1},\alpha_{2}}(t)\right).
\end{equation}
Using (\ref{covxx}), (\ref{Covyx}) and (\ref{covyxj}) in (\ref{R}), we get
\begin{align}\label{fin}
\operatorname{Cov}\left(R^{\alpha_{1},\alpha_{2}}(s), R^{\alpha_{1},\alpha_{2}}(t)\right)&=\lambda^{2}(1+\rho)^{2}(\mathbb{E}(X_1))^{2}   \operatorname{Cov}\left(Y_{\alpha_{1},\alpha_{2}}(s), Y_{\alpha_{1},\alpha_{2}}(t)\right)+\mathbb{E}( X_{1}^{2})\mathbb{E}\left( N^{\alpha_{1},\alpha_{2}}(s)\right)\nonumber\\
&\ \ +\lambda^{2}(\mathbb{E}(X_1))^{2} \operatorname{Cov}\left(Y_{\alpha_{1},\alpha_{2}}(s), Y_{\alpha_{1},\alpha_{2}}(t)\right)\nonumber\\
&\ \ -2\lambda^{2}(1+\rho)(\mathbb{E}(X_1))^{2} \operatorname{Cov}\left(Y_{\alpha_{1},\alpha_{2}}(s), Y_{\alpha_{1},\alpha_{2}}(t)\right)\nonumber\\
&=\lambda^{2} \rho^{2}(\mathbb{E}(X_1))^{2}  \operatorname{Cov}\left(Y_{\alpha_{1},\alpha_{2}}(s), Y_{\alpha_{1},\alpha_{2}}(t)\right)+\mathbb{E}\left( X_{1}^{2}\right)\mathbb{E}\left(N^{\alpha_{1},\alpha_{2}}(s)\right).
\end{align}
	
Put $s=t$ in (\ref{fin}) to get the required variance of MFRP. This completes the proof.
\end{proof}
\begin{theorem}
The risk process $\{R^{\alpha_{1},\alpha_{2}}(t)\}_{t\ge0}$ exhibits the LRD property.
\end{theorem}
\begin{proof}
For fixed $s$ and large $t$, it follows by Proposition \ref{covr} that
\begin{equation*}
\operatorname{Corr}\left(R^{\alpha_{1},\alpha_{2}}(s), R^{\alpha_{1},\alpha_{2}}(t)\right)=\dfrac{\mu^{2} \lambda^{2} \rho^{2} \operatorname{Cov}\left(Y_{\alpha_{1},\alpha_{2}}(s), Y_{\alpha_{1},\alpha_{2}}(t)\right)+\mathbb{E}\left( X_{1}^{2}\right)\mathbb{E}\left( N^{\alpha_{1},\alpha_{2}}(s)\right)}{\sqrt{\operatorname{Var}\left(R^{\alpha_{1},\alpha_{2}}(s)\right)} \sqrt{\mu^{2} \lambda^{2} \rho^{2} \operatorname{Var}\left(Y_{\alpha_{1},\alpha_{2}}(t)\right)+\mathbb{E}\left( X_{1}^{2}\right)\mathbb{E}\left( N^{\alpha_{1},\alpha_{2}}(t)\right)}}.
\end{equation*}
Using (\ref{Tau}), (\ref{3.5}), (\ref{co}) and  (\ref{E}) in the above equation, we get
\begin{align*}
\operatorname{Corr}&\left(R^{\alpha_{1},\alpha_{2}}(s), R^{\alpha_{1},\alpha_{2}}(t)\right)\\
&\sim \dfrac{\mu^{2} \lambda^{2} \rho^{2}s^{2\alpha_{1}}C_{1}^{-2} E_{\alpha_{1}-\alpha_{2}, 2\alpha_{1}+1}^{2}\left(-C_{2}s^{\alpha_{1}-\alpha_{2}}/C_{1}\right)+\mathbb{E} \left(X_{1}^{2}\right)\mathbb{E}\left( N^{\alpha_{1},\alpha_{2}}(s)\right)}{\sqrt{\operatorname{Var}\left(R^{\alpha_{1},\alpha_{2}}(s)\right)}\sqrt{\dfrac{\mu^{2} \lambda^{2} \rho^{2}t^{2\alpha_{2}}}{C_{2}^{2}}\left(\dfrac{2}{\Gamma(2\alpha_{2}+1)}-\dfrac{1}{\left(\Gamma(\alpha_{2}+1)\right)^{2}}\right)+\dfrac{\mathbb{E}\left( X_{1}^{2}\right)t^{\alpha_{2}}}{C_{2}\Gamma(\alpha_{2}+1)}}}\\
&\sim d_1(s)t^{-\alpha_{2}}.
\end{align*}
As $\alpha_{2}\in(0,1)$, the LRD property of MFRP is established.
\end{proof}
\subsection{Mixed fractional Poissonian noise risk process}
For a fixed $\delta>0$, the increments $Z^{\alpha_{1},\alpha_{2}}_{\delta}(t)$, $t\ge0$, of the MFRP $\{R^{\alpha_{1}, \alpha_{2}}(t)\}_{t\ge0}$ is defined as
\begin{equation}\label{qlq1}
Z^{\alpha_{1},\alpha_{2}}_{\delta}(t)=R^{\alpha_{1}, \alpha_{2}}(t+\delta)-R^{\alpha_{1}, \alpha_{2}}(t).
\end{equation}
The process  $Z^{\alpha_{1},\alpha_{2}}_{\delta}\coloneqq \{Z^{\alpha_{1},\alpha_{2}}_{\delta}(t)\}_{t\ge0}$ is called the mixed fractional Poissonian noise risk process (MFPNRP). Next, we show that the MFPNRP exhibits the SRD property.
\begin{theorem}\label{varbgfff}
The MFPNRP $Z^{\alpha_{1},\alpha_{2}}_{\delta}$ has the SRD property.
\end{theorem}
\begin{proof}
Let $s\ge0$ be fixed such that $0\le s+\delta\le t$. From Proposition \ref{covr}, we have
\begin{align*}\label{covz}
\operatorname{Cov}(Z^{\alpha_{1},\alpha_{2}}_{\delta}(s),Z^{\alpha_{1},\alpha_{2}}_{\delta}(t))&=\operatorname{Cov}\left(R^{\alpha_{1}, \alpha_{2}}(s+\delta)-R^{\alpha_{1}, \alpha_{2}}(s),R^{\alpha_{1}, \alpha_{2}}(t+\delta)-R^{\alpha_{1}, \alpha_{2}}(t)\right)\\
&=\operatorname{Cov}\left(R^{\alpha_{1}, \alpha_{2}}(s+\delta),R^{\alpha_{1}, \alpha_{2}}(t+\delta)\right)+\operatorname{Cov}\left(R^{\alpha_{1}, \alpha_{2}}(s),R^{\alpha_{1}, \alpha_{2}}(t)\right)\\
&\ \ \  -\operatorname{Cov}\left(R^{\alpha_{1}, \alpha_{2}}(s+\delta),R^{\alpha_{1}, \alpha_{2}}(t)\right)-\operatorname{Cov}\left(R^{\alpha_{1}, \alpha_{2}}(s),R^{\alpha_{1}, \alpha_{2}}(t+\delta)\right)\\
&=\mu^{2} \lambda^{2} \rho^{2}\big\{\operatorname{Cov}\left(Y_{\alpha_{1},\alpha_{2}}(s), Y_{\alpha_{1},\alpha_{2}}(t)\right)-\operatorname{Cov}\left(Y_{\alpha_{1},\alpha_{2}}(s+\delta), Y_{\alpha_{1},\alpha_{2}}(t)\right)\\
&\ \ -\operatorname{Cov}\left(Y_{\alpha_{1},\alpha_{2}}(s), Y_{\alpha_{1},\alpha_{2}}(t+\delta)\right)+\operatorname{Cov}\left(Y_{\alpha_{1},\alpha_{2}}(s+\delta), Y_{\alpha_{1},\alpha_{2}}(t+\delta)\right)\big\}.
\end{align*}
For large $t$, the asymptotic behaviour of the covariance of mixed inverse stable subordinator is given by (see  Kataria and Khandakar (2019), Eq. (4.3))
\begin{equation}\label{q1q1qqa}
\operatorname{Cov}\left(Y_{\alpha_{1},\alpha_{2}}(s), Y_{\alpha_{1},\alpha_{2}}(t)\right)\sim\frac{s^{2\alpha_{1}}}{C_{1}^{2}} E_{\alpha_{1}-\alpha_{2}, 2\alpha_{1}+1}^{2}\left(-C_{2}s^{\alpha_{1}-\alpha_{2}}/C_{1}\right)-t^{\alpha_2-1}K(s),
\end{equation}
where
\begin{equation*}
K(s)=\frac{s^{\alpha_1+1}}{C_{1}C_{2}\Gamma(\alpha_{2})}\sum_{k=0}^{\infty}\frac{\left(k(\alpha_{1}-\alpha_{2})+\alpha_{1}\right)(-C_{2}s^{(\alpha_{1}-\alpha_{2})}/C_{1})^{k}}{\Gamma\left(k(\alpha_{1}-\alpha_{2})+\alpha_{1}+2\right)}.
\end{equation*}
Using (\ref{q1q1qqa}), we get
\begin{align}\label{covzst}
\operatorname{Cov}(Z^{\alpha_{1},\alpha_{2}}_{\delta}(s),Z^{\alpha_{1},\alpha_{2}}_{\delta}(t))&\sim\mu^{2}\lambda^2\rho^{2}\big(t^{\alpha_2-1}K(s+\delta)+(t+\delta)^{\alpha_2-1}K(s)\nonumber\\
&\hspace{2cm} -(t+\delta)^{\alpha_2-1}K(s+\delta)-t^{\alpha_2-1}K(s)\big)\nonumber\\
&=\mu^{2} \lambda^{2} \rho^{2}(K(s+\delta)-K(s))(t^{\alpha_2-1}-(t+\delta)^{\alpha_2-1})\nonumber\\
&=\mu^{2} \lambda^{2} \rho^{2}\left(K(s+\delta)-K(s)\right)t^{\alpha_2-1}\left(1-\left(1+\frac{\delta }{t}\right)^{\alpha_2-1}\right)\nonumber\\
&\sim(1-\alpha_2)\delta\mu^{2} \lambda^{2} \rho^{2}(K(s+\delta)-K(s))t^{\alpha_2-2}.
\end{align}
The following result will be used (see Kataria and Khandakar (2019), Eq. (4.5) and Eq. (4.6)):
\begin{align}\label{covt}
\operatorname{Cov}\left(Y_{\alpha_{1},\alpha_{2}}(t), Y_{\alpha_{1},\alpha_{2}}(t+\delta)\right)&\sim\frac{t^{2\alpha_{1}}}{C_{1}^{2}} E_{\alpha_{1}-\alpha_{2}, 2\alpha_{1}+1}^{2}\left(-C_{2}t^{\alpha_{1}-\alpha_{2}}/C_{1}\right)-U_{\alpha_{1}, \alpha_{2}}(t)U_{\alpha_{1}, \alpha_{2}}(t+\delta)\nonumber\\
&\ +\frac{(t+\delta)^{\alpha_{1}+\alpha_{2}}}{C_{1}C_{2}}E_{\alpha_{1}-\alpha_{2}, \alpha_{1}+\alpha_{2}+1}\left(-C_{2}(t+\delta)^{\alpha_{1}-\alpha_{2}}/C_{1}\right).
\end{align}
\begin{align}\label{varz}
\operatorname{Var}(Z^{\alpha_{1},\alpha_{2}}_{\delta}(t))&=\operatorname{Var}(R^{\alpha_{1}, \alpha_{2}}(t+\delta))+\operatorname{Var}(R^{\alpha_{1}, \alpha_{2}}(t))-2\operatorname{Cov}\left(R^{\alpha_{1}, \alpha_{2}}(t),R^{\alpha_{1}, \alpha_{2}}(t+\delta)\right)\nonumber\\
&=\mu^{2} \lambda^{2} \rho^{2}\left(\operatorname{Var}\left(Y_{\alpha_{1},\alpha_{2}}(t+\delta)\right)+\operatorname{Var}\left(Y_{\alpha_{1},\alpha_{2}}(t)\right)\right)+\mathbb{E}\left( X_{1}^{2}\right)\big(\mathbb{E}\left(N^{\alpha_{1},\alpha_{2}}(t+\delta)\right)\nonumber\\ 
&\ -\mathbb{E}\left(N^{\alpha_{1},\alpha_{2}}(t)\right)\big)-2\mu^{2} \lambda^{2} \rho^{2} \operatorname{Cov}\left(Y_{\alpha_{1},\alpha_{2}}(t), Y_{\alpha_{1},\alpha_{2}}(t+\delta)\right),\ (\mathrm{by\ Proposition}\ \ref{covr})\ \nonumber\\
&\sim\frac{\mu^{2} \lambda^{2} \rho^{2}}{C_{2}^{2}}\left(\frac{2}{\Gamma(2\alpha_{2}+1)}-\frac{1}{\left(\Gamma(\alpha_{2}+1)\right)^{2}}\right)\left((t+\delta)^{2\alpha_{2}}+t^{2\alpha_{2}}\right)\nonumber\\
&\ +\dfrac{\lambda\mathbb{E}\left( X_{1}^{2}\right)}{C_{2} \Gamma\left(\alpha_{2}+1\right)}\left((t+\delta)^{\alpha_{2}}-t^{\alpha_{2}}\right)-2\mu^{2} \lambda^{2} \rho^{2}\bigg(\frac{t^{2\alpha_{1}}}{C_{1}^{2}} E_{\alpha_{1}-\alpha_{2}, 2\alpha_{1}+1}^{2}\left(-C_{2}t^{\alpha_{1}-\alpha_{2}}/C_{1}\right)\nonumber\\
& \ -U_{\alpha_{1}, \alpha_{2}}(t)U_{\alpha_{1}, \alpha_{2}}(t+\delta)
 +\frac{(t+\delta)^{\alpha_{1}+\alpha_{2}}}{C_{1}C_{2}}E_{\alpha_{1}-\alpha_{2}, \alpha_{1}+\alpha_{2}+1}\left(-C_{2}(t+\delta)^{\alpha_{1}-\alpha_{2}}/C_{1}\right)\bigg),\nonumber\\
&\hspace{7cm}(\mathrm{using}\ (\ref{Tau}),\ (\ref{3.5}),\ (\ref{E}) \ \mathrm{and} \  (\ref{covt}))\nonumber\\
&\sim\frac{\mu^{2} \lambda^{2} \rho^{2}}{C_{2}^{2}}\left(\frac{2}{\Gamma(2\alpha_{2}+1)}-\frac{1}{\left(\Gamma(\alpha_{2}+1)\right)^{2}}\right)t^{2\alpha_{2}}\left(\left(1+\dfrac{\delta}{t}\right)^{2\alpha_{2}}+1\right)\nonumber\\
&\ +\dfrac{\lambda\mathbb{E}\left( X_{1}^{2}\right)}{C_{2} \Gamma\left(\alpha_{2}+1\right)}t^{\alpha_{2}}\left(\left(1+\dfrac{\delta}{t}\right)^{\alpha_{2}}-1\right)-2\mu^{2}\lambda^{2}\rho^{2}\bigg(\frac{t^{2\alpha_{2}}}{C_{2}^{2}\Gamma(2\alpha_{2}+1)}\nonumber\\
&\ -\frac{t^{\alpha_2}(t+\delta)^{\alpha_2}}{C_{2}^{2}\left(\Gamma(\alpha_{2}+1)\right)^{2}}+\frac{(t+\delta)^{2\alpha_{2}}}{C_{2}^{2}\Gamma(2\alpha_{2}+1)}\bigg),\ (\mathrm{using}\ (\ref{3m})\ \mathrm{and}\ (\ref{Tau}))\nonumber\\
&\sim\frac{\mu^{2} \lambda^{2} \rho^{2}}{C_{2}^{2}}\left(\frac{2}{\Gamma(2\alpha_{2}+1)}-\frac{1}{\left(\Gamma(\alpha_{2}+1)\right)^{2}}\right)2t^{2\alpha_{2}}(1+\alpha_2\delta t^{-1})+\dfrac{\lambda\alpha_2\delta\mathbb{E}\left( X_{1}^{2}\right)t^{\alpha_{2}-1}}{C_{2} \Gamma\left(\alpha_{2}+1\right)}\nonumber\\
& \ -\dfrac{2\mu^{2}\lambda^{2}\rho^{2}t^{2\alpha_{2}}}{C_{2}^{2}}\left(\frac{1}{\Gamma(2\alpha_{2}+1)}-\frac{1+\delta\alpha_{2}t^{-1}}{\left(\Gamma(\alpha_{2}+1)\right)^{2}}+\frac{1+2\alpha_{2}\delta t^{-1}}{\Gamma(2\alpha_{2}+1)}\right)\nonumber\\
&= \dfrac{\lambda\alpha_2\delta\mathbb{E}\left( X_{1}^{2}\right)t^{\alpha_{2}-1}}{C_{2} \Gamma\left(\alpha_{2}+1\right)}.
\end{align}
From (\ref{covzst}) and (\ref{varz}), it follows that
\begin{align*}
\operatorname{Corr}\left(Z^{\alpha_{1},\alpha_{2}}_{\delta}(s),Z^{\alpha_{1},\alpha_{2}}_{\delta}(t)\right)&=\dfrac{\operatorname{Cov}\left(Z^{\alpha_{1},\alpha_{2}}_{\delta}(s),Z^{\alpha_{1},\alpha_{2}}_{\delta}(t)\right)}{\sqrt{\operatorname{Var}Z^{\alpha_{1},\alpha_{2}}_{\delta}(s)}\sqrt{\operatorname{Var}Z^{\alpha_{1},\alpha_{2}}_{\delta}(t)}}\\
&\sim\dfrac{(1-\alpha_2)\delta\mu^{2} \lambda^{2} \rho^{2}(K(s+\delta)-K(s))t^{\alpha_2-2}}{\sqrt{\operatorname{Var}Z^{\alpha_{1},\alpha_{2}}_{\delta}(s)}\sqrt{\dfrac{\lambda\alpha_2\delta\mathbb{E}\left( X_{1}^{2}\right)t^{\alpha_{2}-1}}{C_{2} \Gamma\left(\alpha_{2}+1\right)}}},\ \ t\rightarrow\infty.
\end{align*}
Thus,
\begin{equation*}
\operatorname{Corr}\left(Z^{\alpha_{1},\alpha_{2}}_{\delta}(s),Z^{\alpha_{1},\alpha_{2}}_{\delta}(t)\right)\sim d_2(s)t^{-(3-\alpha_{2})/2},\ \ \mathrm{as}\ t\rightarrow\infty.
\end{equation*}
As $1<(3-\alpha_{2})/2<1.5$, the SRD property of MFPNRP is established.
\end{proof}
Kumar {\it et al.} (2019) considered a variant of $\{R^{\alpha}(t)\}_{t\ge0}$ by taking the expected value of the inverse stable subordinator $Y_\alpha(t)$ in its premium. Similarly, we give a variant of the MFRP as follows:
\begin{equation}\label{**}
\tilde{R}^{\alpha_{1},\alpha_{2}}(t)=u+\mu(1+\rho)\lambda U_{{\alpha_{1},\alpha_{2}}}(t)-\sum_{i=1}^{N^{\alpha_{1},\alpha_{2}}(t)}X_{i}, \ \ t\ge0,
\end{equation}
where $U_{{\alpha_{1},\alpha_{2}}}(t)$ is given by (\ref{7}). As in the case of MFRP, here we have $u$ as the initial capital, $\mu$ as the constant premium rate and $\left\{X_{i}\right\}_{i \geq 1}$ as a sequence of iid random variables which is independent of the MFPP. For the non-homogeneous case, we have
\begin{equation*}
\tilde{R}_{\Lambda}^{\alpha_{1},\alpha_{2}}(t)=u+\mu
\mathbb{E}\left(\Lambda(Y_{\alpha_{1},\alpha_{2}}(t))\right)-\sum_{i=1}^{N^{\alpha_{1},\alpha_{2}}_{\Lambda}(t)}X_{i}, \ \ t\ge0.
\end{equation*}
\section{Mixed Fractional Risk Process-II}
In this section, we introduce a second version of the fractional risk process where the number of claims are modelled using the MFPP. Beghin and Macci (2013) introduced  the following fractional version of the classical risk process:
\begin{equation}\label{beghin}
\bar{R}(t)=u+ct-\sum _{i=1}^{N^{\alpha}(t)}X_{i},  \quad t\ge0.
\end{equation}
In the above process, the $X_{i}$'s are iid positive random variables which are independent of the TFPP $\{N^{\alpha}(t)\}_{t\ge0}$. The premium rate is a  constant $c>0$. The risk process (\ref{beghin}) can be generalized as 
\begin{equation}\label{bar}
\bar{R}^{\alpha_{1},\alpha_{2}}(t)=u+ct-\sum_{i=1}^{N^{\alpha_{1},\alpha_{2}}(t)}X_{i}, \ \ t\ge0.
\end{equation}
 We call the process $\{\bar{R}^{\alpha_{1},\alpha_{2}}(t)\}_{t\ge0}$ as the mixed fractional risk process-II (MFRP-II). Note that in this process also the number of claims are modelled using the MFPP. It differs from the MFRP in terms of the premium. Its expected value is
\begin{equation*}
\mathbb{E}\left(\bar{R}^{\alpha_{1},\alpha_{2}}(t)\right)=u+ct-\mu\lambda U_{{\alpha_{1},\alpha_{2}}}(t),
\end{equation*}
where $U_{{\alpha_{1},\alpha_{2}}}(t)$ is given by (\ref{7}).
For $0\leq s\leq t$, the covariance  of MFRP-II is obtained as follows:
\begin{equation*}
\operatorname{Cov}\left(\bar{R}^{\alpha_{1},\alpha_{2}}(s), \bar{R}^{\alpha_{1},\alpha_{2}}(t)\right)=\operatorname{Cov}\left(\sum_{j=1}^{N^{\alpha_{1},\alpha_{2}}(s)} X_{j}, \sum_{i=1}^{N^{\alpha_{1},\alpha_{2}}(t)} X_{i}\right)=I(s,t).
\end{equation*}
From (\ref{covxx}), we get
\begin{equation}\label{covii}
\operatorname{Cov}\left(\bar{R}^{\alpha_{1},\alpha_{2}}(s), \bar{R}^{\alpha_{1},\alpha_{2}}(t)\right)=\mathbb{E}(X_{1}^{2})\mathbb{E}\big(N^{\alpha_{1},\alpha_{2}}(s)\big)+\lambda^{2}(\mathbb{E}(X_1))^{2}\operatorname{Cov}\left(Y_{\alpha_{1},\alpha_{2}}(s), Y_{\alpha_{1},\alpha_{2}}(t)\right).
\end{equation}
Substituting $s=t$ gives its variance as
\begin{equation}\label{varii}
\operatorname{Var}\left(\bar{R}^{\alpha_{1},\alpha_{2}}(t)\right)=\mathbb{E}\left( X_{1}^{2}\right)\mathbb{E}\left(N^{\alpha_{1},\alpha_{2}}(t)\right)+\lambda^{2}(\mathbb{E}(X_1))^{2} \operatorname{Var}\left(Y_{\alpha_{1},\alpha_{2}}(t)\right).
\end{equation}
\begin{remark}
It is important to note that the expressions obtained for the variance and covariance of MFRP-II are almost similar to that of the MFRP. The difference is that of the safety loading factor $\rho$. Thus, as in the case of MFRP, it follows that the MFRP-II also exhibits the LRD property. For the same reason, the related increment process for the MFRP-II defined as 
\begin{equation}\label{qlq2}
\bar{Z}^{\alpha_{1},\alpha_{2}}_{\delta}(t)\coloneqq\bar{R}^{\alpha_{1}, \alpha_{2}}(t+\delta)-\bar{R}^{\alpha_{1}, \alpha_{2}}(t),
\end{equation}
has the SRD property.
\end{remark}
\subsection{Ruin Probabilities}
Biard and Saussereau (2014) derived some expressions of the ruin probabilities for fractional risk process defined in (\ref{beghin}). They consider the light-tailed and heavy-tailed distributions of claim sizes. Here, we obtain the ruin probabilities of MFRP-II in the presence of exponential and subexponential claim sizes.
\subsubsection{Exponentially distributed claim sizes}
Assume that the claim sizes $X_i$'s in the MFRP-II defined in (\ref{bar}) are exponentially distributed with parameter $\mu>0$. Its ruin time $T$ is defined as 
\begin{equation}\label{qaxs1z}
T\coloneqq\inf\{t>0:\bar{R}^{\alpha_{1},\alpha_{2}}(t)<0\}.
\end{equation}
Note that, $T=\infty$ if $\bar{R}^{\alpha_{1},\alpha_{2}}(t)>0$ for all $t$.
Using Theorem 1 of Borokov and Dickson (2008) with interarrival time density   $f_{W}^{\alpha_{1},\alpha_{2}}(t)$ given in (\ref{density}), the probability density function $f_{T}(t)$ of the ruin time $T$ is given by
\begin{equation*}
f_{T}(t)=e^{-\mu(u+ct)}\sum_{n=0}^{\infty}\dfrac{\mu^{n}(u+ct)^{n-1}}{n!}\left(u+\frac{ct}{n+1}\right)(f_{W}^{\alpha_{1},\alpha_{2}})^{*(n+1)}(t),
\end{equation*}
where $(f_{W}^{\alpha_{1},\alpha_{2}})^{*(n+1)}(t)$ denotes the $(n+1)$-fold convolution of the function $f_{W}^{\alpha_{1},\alpha_{2}}(t)$.

The ruin probability $\psi_{u}(t)$ of the MFRP-II in finite time $t$ is
$\psi_{u}(t)=\mathrm{Pr}\{T\le t<\infty\}$. Its Laplace transform  $\tilde{\psi}_{u}(s)$  can be obtained by Theorem 1 of Malinovskii (1998) using (\ref{laplace}) as
\begin{equation*}
\tilde{\psi}_{u}(s)=s^{-1}y(s) \exp \{-u \mu(1-y(s))\}, \quad s>0,
\end{equation*}
where $y(s)$ is a solution of 
\begin{equation*}
y(s)=\lambda/\left(C_{1}(s + c\mu(1-y(s)))^{\alpha_{1}}+C_{2}(s +c \mu(1-y(s)))^{\alpha_{2}}+\lambda\right). 
\end{equation*}

\subsubsection{Subexponentially distributed claim sizes}
Let the distribution function $F_{X_{1}}(t)=\mathrm{Pr}\{X_1\leq t\}$ of the claim sizes $X_i$'s in the MFRP-II be subexponential, that is, $\lim\limits_{t\to \infty}(1-F_{X_{1}}^{*2}(t))/(1-F_{X_{1}}(t))=2$. The following result related to subexponential distribution will be used (see Asmussen and Albrecher (2010)).
\begin{lemma}\label{qlap22}
Let $\{X_{i}\}_{i\ge1}$ be a sequence of iid random variables having subexponential distribution $\bar{F}_{X_{1}}(t)=\mathrm{Pr}\{X_{1}>t\}$. Then, as $t\to \infty$ it holds that
\begin{equation*}
\mathrm{Pr}\left\{\sum_{i=1}^{N}X_{i}>t\right\}\sim \mathbb{E}(N)\bar{F}_{X_{1}}(t),
\end{equation*}
where $N$ is an integer valued random variable with $\mathbb{E}(z^{N})<\infty$ for some $z>1$ and it is independent of $\{X_{i}\}_{i\ge1}$.
\end{lemma}
The following inequalities holds for the finite time ruin probability of MFRP-II:
\begin{align}
\mathrm{Pr}\left\{\sum_{i=1}^{N^{\alpha_{1},\alpha_{2}}(t)}X_{i}>u+ct\right\}&\le\mathrm{Pr}\left\{\sum_{i=1}^{N^{\alpha_{1},\alpha_{2}}(t^*)}X_{i}>u+ct^*\ \mathrm{for\ some}\ t^* \leq t<\infty\right\}\nonumber\\
&\le\mathrm{Pr}\left\{\sum_{i=1}^{N^{\alpha_{1},\alpha_{2}}(t)}X_{i}>u\right\}.\label{ine} 
\end{align}
The pgf of MFPP is given in (\ref{qmmmqa1}) which is finite for some $z>1$. Applying Lemma \ref{qlap22} to (\ref{ine}), we get
\begin{equation*}
\mathrm{Pr}\left\{\bar{R}^{\alpha_{1},\alpha_{2}}(t^*)<0\ \mathrm{for\ some}\ t^* \leq t<\infty\right\}\sim\mathbb{E}\left(N^{\alpha_{1},\alpha_{2}}(t)\right)\bar{F}_{X_{1}}(u)\ \ \mathrm{as}\ u\rightarrow \infty.
\end{equation*}
Here, we have used that $\bar{F}_{X_{1}}(u+t)\sim\bar{F}_{X_{1}}(u)$ as $u\rightarrow\infty$. Thus, the ruin probability in finite time has the following asymptotic behaviour as $u\rightarrow \infty$
\begin{equation*}
\mathrm{Pr}\left\{\bar{R}^{\alpha_{1},\alpha_{2}}(t^*)<0\ \mathrm{for\ some}\ t^* \leq t<\infty\right\}\sim\lambda U_{{\alpha_{1},\alpha_{2}}}(t)\bar{F}_{X_{1}}(u),
\end{equation*}
where $U_{{\alpha_{1},\alpha_{2}}}(t)$ is given by (\ref{7}).

\printbibliography
\end{document}